\documentclass{amsart}
\usepackage{amsfonts, amssymb, amsmath, verbatim, amsthm, mathrsfs, amscd, enumerate, fancyhdr, caption, framed, ulem, subfigure}
\usepackage[utf8]{inputenc}
\usepackage{enumitem}

\makeatletter
\@namedef{subjclassname@2020}{\textup{2020} Mathematics Subject Classification}

\usepackage[dvipsnames]{xcolor}
\usepackage{tikz}
\usetikzlibrary{positioning,intersections, calc, arrows,decorations.markings}
\usepackage{tikz-3dplot}
\usepackage{ulem}

\usepackage{hyperref}
 \usepackage{enumitem}
\usepackage{capt-of}
\usepackage{relsize}

\newtheorem{theorem}{Theorem}[section]
\newtheorem{lemma}[theorem]{Lemma}
\newtheorem{rmk}[theorem]{Remark}
\newtheorem{proposition}[theorem]{Proposition}
\newtheorem*{theorem*}{Theorem}

\newtheorem*{conjecture*}{Conjecture}

\newtheorem{corollary}[theorem]{Corollary}

\newtheorem*{claim*}{Claim}
\theoremstyle{definition}

\newtheorem*{goal*}{Goal}

\theoremstyle{remark}

\numberwithin{equation}{section}

\def\supp{{\rm supp\,}}

\newcommand{\B}{{\mathbb B}}

\newcommand{\Be}{\begin{equation}}
\newcommand{\Ee}{\end{equation}}
\newcommand{\Bea}{\begin{eqnarray}}
\newcommand{\Eea}{\end{eqnarray}}
\newcommand{\Bel}{\begin{align}}
\newcommand{\Eel}{\end{align}}
\newcommand{\Beas}{\begin{eqnarray*}}
\newcommand{\Eeas}{\end{eqnarray*}}
\newcommand{\Benu}{\begin{enumerate}}
\newcommand{\Eenu}{\end{enumerate}}
\newcommand{\Bi}{\begin{itemize}}
\newcommand{\Ei}{\end{itemize}}

\newcommand{\dist}{\operatorname{dist}}

\begin{document}
\title[{Endpoint maximal estimates}]{Endpoint estimates for maximal operators associated to the wave equation}

\author{Chu-hee Cho}
\author{Sanghyuk Lee}
\address{Department of Mathematical Sciences and RIM, Seoul National University, Seoul 08826, Republic of Korea}
\email{akilus@snu.ac.kr}

\email{shklee@snu.ac.kr}

\author{Wenjuan Li}
\address{School of Mathematics and Statistics, Northwestern Polytechnical University, Xi’an 710129, China}
\email{liwj@nwpu.edu.cn}

\makeatletter
\@namedef{subjclassname@2020}{\textup{2020} Mathematics Subject Classification}
\makeatother
\subjclass[2020]{Primary 35L05,  Secondary 42B37}
\keywords{Wave operator, maximal estimate}

\begin{abstract} We consider  the $H^{s}$--$L^q$ maximal  estimates associated to the wave operator
\begin{equation*}
	e^{ it\sqrt{-\Delta}}f(x) = \frac{1}{(2\pi)^d}\int_{\mathbb{R}^d} e^{i(x \cdot \xi \, + t|\xi|)} \widehat{f}(\xi\,) d\xi.
\end{equation*}
Rogers--Villarroya proved  $H^{s}$--$L^q$ estimates for  the maximal operator $f\mapsto$ $\sup_{t} |e^{ it\sqrt{-\Delta}}f|$
up to the critical Sobolev exponents $s_c(q,d)$. However, the endpoint case estimates for the critical exponent $s=s_c(q,d)$ have remained open so far. 
We obtain the endpoint $H^{s_c(q,d)}$--$L^q$ bounds  on the  maximal operator $f\mapsto \sup_{t} |e^{ it\sqrt{-\Delta}}f|$. 
We also prove that  several different forms of the maximal estimates considered by Rogers--Villarroya are basically  equivalent to each other. 
\end{abstract}

\maketitle

\section{Introduction}
For $d\ge 2$, let us consider the wave equation 
	\[ \partial_{t}^2u(x,t)-\triangle u(x,t) =0, \quad (x,t) \in \mathbb{R}^{d}\times \mathbb{R}^+\]
with initial data $u(x,0)=f$ and $\partial_t u(x,0)=0$, whose formal solution 
 is given by 
\begin{align*}
	u(x,t)
	=\frac{1}{2}\Bigl(e^{ it\sqrt{-\Delta}}f(x)+e^{ -it\sqrt{-\Delta}}f(x) \Bigl).
\end{align*}
In this paper, we are concerned with several types of maximal estimates for the wave operator $f\mapsto e^{ it\sqrt{-\Delta}}f$.  Such estimates have been utilized  to study  
pointwise convergence of the solution $u(\cdot,t)$ to the initial data $f$  as $t\to 0$.  They also find applications in the study of the linear and nonlinear wave equations. Specifically, we focus on local and global (in time) maximal operators in  $L^q$ or $L^q_{loc}$ space, 
as investigated in  \cite{RV} (see the estimates \eqref{max}, \eqref{local-local}, \eqref{global-local}, and \eqref{global-global0} below).

\subsection{Local in time estimate.} We first consider the local in time maximal estimate  
\begin{equation}\label{max}
	\big\| \sup_{t \in (0,1)} \big|e^{it\sqrt{-\Delta}}f\big|\big\|_{L^{q}(\mathbb{R}^d)} \le C_{d,q,s} \|f\|_{{H}^s(\mathbb{R}^d)},
\end{equation}
which was studied by various authors. Here ${H}^s(\mathbb{R}^d)$ denotes the inhomogeneous Sobolev space of order $s$. 
Cowling \cite{C83} proved that \eqref{max} holds for $q=2$ and $s>1/2$,  and it was shown by  Walther \cite{Wa99} that the regularity requirement $s>1/2$ is sharp.
 Later, Rogers and Villarroya \cite{RV} extended the estimate \eqref{max} to general \( q \in (2, \infty] \). It should be noted  that the estimate \eqref{max} is valid only for \( q \geq 2 \), as  can be justified by making use of  translation invariance of the wave operator.

Let us set $q_\circ(d)= \frac{2(d+1)}{d-1}$ and 
\[ s_c(q,d): =\begin{cases} \frac {d+1}4-\frac{d-1}{2q}, \ & 2\le q< q_\circ(d),
	\\[4pt]
	\ \ \  \frac d2-\frac dq,\ &  q_\circ(d)\le q\le \infty.
\end{cases} \]
It was shown in \cite{RV} that \eqref{max} holds if $s>s_c(q,d)$, and \eqref{max} fails to hold if $s<s_c(q,d)$.  
However, most of the estimates for the endpoint case $s=s_c(q,d)$  have remained  open until recently.

Previously, there were very few results addressing the estimate of critical index \( s = s_c(q, d) \). For \( d = 3 \) and \(q \in [6, \infty) \),  Beceanu and Goldberg \cite{BG12}, using the reversed-norm Strichartz inequality and the \( TT^* \) argument, proved that  
\[
\|\sup_{t \in \mathbb{R}} |e^{it\sqrt{-\Delta}} f|\|_{L^{q}(\mathbb{R}^3)} \leq C \|f\|_{\dot{H}^{s_c(q, d)}(\mathbb{R}^3)}.
\]
This estimate is in fact equivalent to the seemingly weaker estimate \eqref{max} with $s= s_c(q,3)$;  see Proposition \ref{prop:local-global to global} below.  From the perspective of the global in time maximal estimate (e.g., see \eqref{global-global0}), their result does not cover the case \( q \in [4, 6) \). Later, Machihara \cite{SM13} obtained  the estimate \eqref{global-global0}  for \( d = 3 \), \( q = 4 \), and \( s_c(4,3) = {3}/{4} \) with radial functions \( f \) satisfying a certain monotonicity property.

In this paper, we provide an almost complete answer to the endpoint estimates. We obtain strong type estimates for the critical case $s=s_c(q,d)$ with $q\in(2, \infty)\setminus\{q_\circ(d)\}$.

\begin{theorem}\label{thm:max}
	Let $q\in(2, \infty)\setminus\{q_\circ(d)\} $. Then, we have the estimate 
	\begin{equation}\label{max-c}
	\big\| \sup_{t \in (0,1)} \big|e^{it\sqrt{-\Delta}}f\big|\big\|_{L^{q}(\mathbb{R}^d)} \le C_{d,q,s} \|f\|_{{H}^{s_c(q,d)}(\mathbb{R}^d)}. 
\end{equation}
\end{theorem}

The case $q= q_\circ(d)$  remains open. Unfortunately, it looks unlikely that the method in this paper recovers this case. 
However,  for $q= q_\circ(d)$, we have the estimate 
	\Be  \label{max2}  \| \sup_{t\in (1,2)} |e^{it\sqrt{-\Delta}} f|\|_{L^{q_\circ(d),\infty}(\mathbb R^d)} \le C_{d,q,s} \|f\|_{H^{s_c(q_\circ(d),d),1}(\mathbb R^d)},  \Ee 
	where $H^{s,1}(\mathbb R^d)$ denotes the nonhomogeneous Sobolev--Lorentz space with the norm  $\|f\|_{H^{s,1}(\mathbb R^d)}=\|(1-\Delta)^{s/2}f\|_{L^{2,1}}.$
We refer the reader to Section \ref{caseqc} for more details regarding the estimate \eqref{max2}. 

We also make a couple of remarks on the cases $q=2$ and $q=\infty$.  In fact, it was proved by Ham--Ko--Lee \cite[Lemma A.2]{HKL}  that 
	\[  \| \sup_{t\in (0,1)} |e^{it\sqrt{-\Delta}} f|\|_{L^{q,\infty}(B^d(0,1))} \le C\|f\|_{H^{1/2}(\mathbb R^d)}  \] 
fails for any $q\ge 1$. In particular,  the weak type estimate is not possible  for $q=2$.  It is easy to see that 
the estimate \eqref{max}  fails for $q=\infty$ and $s=s_c(\infty, d)=d/2$   from failure of  the Sobolev imbedding  $H^{d/2} \hookrightarrow L^\infty$ (see, for example, \cite[Remark 3]{FW}).  

To establish the endpoint estimates in Theorem \ref{thm:max}, we consider not only the \(H^s\)--\(L^q\) estimates but also the broader framework of \(L^p_s\)--\(L^q\) estimates. Specifically, the desired \(H^s\)--\(L^q\) estimates will be derived by proving various optimal (endpoint case) \(L^p_s\)--\(L^q\) estimates. Here \(L^p_s\) denotes the $L^p$ Sobolev spaces of order $s$ so that \(L^2_s=H^s\).  For this purpose, we adopt the approach developed in \cite{L}, which leverages the bilinear restriction estimates for the cone due to Wolff \cite{W} and Tao \cite{T}.    As is evident to experts in the field, various endpoint maximal estimates for \(f \in L_s^p\) over an extended range can also be shown by using currently available results on local smoothing estimates for the wave operator. However, a complete resolution of this problem is closely linked to the local smoothing estimates with optimal regularity, which remain unresolved to date except for some special cases. We do not attempt to address this matter here.

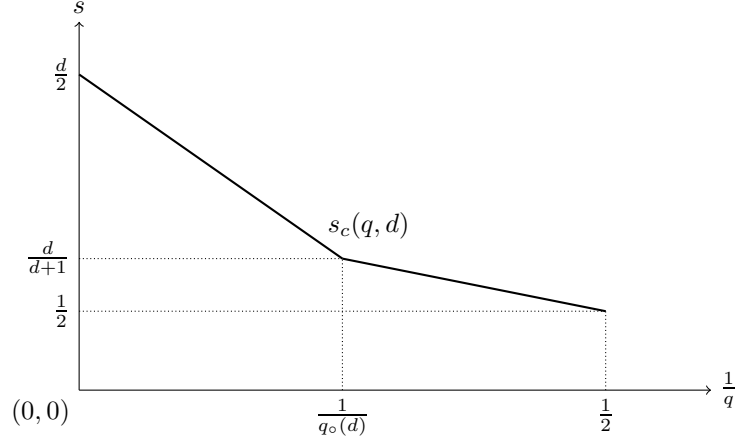
\begin{figure}[t]
\begin{center}
\begin{tikzpicture}[scale=7]
    \draw[->] (0,0) -- (1.2,0) node[right] {\(\frac1q\)};
    \draw[->] (0,0) -- (0,0.7) node[above] {\(s\)};
        
     \draw[densely dotted] (1,0) -- (1,0.15);
     \draw[densely dotted] (0.5,0) -- (0.5,0.25);
     \draw[densely dotted] (0,0.25) -- (0.5,0.25);
     \draw[densely dotted] (0,0.15) -- (1,0.15);

    \draw[thick] (0,0.6) -- (0.5,0.25);
    \draw[thick] (1,0.15) -- (0.5,0.25);
    \node[above] at (0.55,0.27) {\(s_c(q,d)\)};

    \node[below left] at (0,0) {\((0,0)\)};
    \node[left] at (0,0.6) {\(\frac d2\)};
    \node[left] at (0,0.25) {\(\frac{d}{d+1}\)};
    \node[left] at (0,0.15) {\(\frac12\)};
    \node[below] at (0.5,0) {\(\frac 1{q_\circ(d)}\)};
    \node[below] at (1,0) {\(\frac12\)};
    
\end{tikzpicture}
\caption{The critical regularity  exponent $s_c(q,d)$ of the estimate \eqref{max}.}
\end{center}
\end{figure}

\subsection{Local time-space maximal estimate}
Several variant forms of the estimate \eqref{max} were also studied in \cite{RV}. These variants can be applied to different purposes depending on the nature of the problem under consideration. However, as will be seen below, they turn out to be essentially equivalent. To the best of our knowledge, such equivalences have not been formally established before.

Let us first consider the local time-space estimate
\begin{equation}\label{local-local}
	\big\| \sup_{t \in (0,1)} \big|e^{it\sqrt{-\Delta}}f\big|\big\|_{L^{q}(B^d(0,1))} \le C_{n,q,s} \|f\|_{H^s(\mathbb{R}^d)}.
\end{equation}
Trivially, this estimate follows from \eqref{max}. However, thanks to finite speed of propagation and translation invariance of  the wave operator, 
 the estimate \eqref{local-local} implies \eqref{max} if $q\ge 2$. Consequently, we have the next proposition. 

\begin{proposition}\label{prop:xlocal to global}
Let $q\ge 2$ and $s\ge 0$. Then, the estimates  \eqref{max}  and  \eqref{local-local} are equivalent.
\end{proposition}

An immediate consequence of this equivalence and Theorem \ref{thm:max} is the following.

\begin{corollary}\label{cor:local-local}
Let $q\in(2, \infty)\setminus\{q_\circ(d)\} $. Then, the  estimate \eqref{local-local}  holds if and only if  $s\ge s_c(q,d)$. 
\end{corollary}

\subsection{Global in time and  local in space maximal estimate}
We also consider the  estimate
\begin{equation}\label{global-local}
	\big\| \sup_{t \in \mathbb R} \big|e^{it\sqrt{-\Delta}}f\big|\big\|_{L^{q}(B^d(0,1))} \le C_{n,q,s} \|f\|_{H^s(\mathbb{R}^d)},
\end{equation}
which is global in time and local in space.   At first glance, the implication from the estimate \eqref{local-local} to \eqref{global-local} is not immediately clear while 
the opposite direction is trivial.  However, it turns out that they are equivalent when $q\ge 2$.
\begin{proposition}\label{prop:local to global}
Let $q\ge 2$ and $s\ge 0$. Then, the estimates \eqref{local-local} and \eqref{global-local} are equivalent.
\end{proposition}

Thanks to the equivalence  and Corollary \ref{cor:local-local}, we obtain the following.

\begin{corollary}\label{cor:global-local}
Let $q\in(2, \infty)\setminus\{q_\circ(d)\} $. Then \eqref{global-local} holds if and only if  $s\ge s_c(q,d)$. 
\end{corollary}

\subsection{Global time-space maximal estimate}
We now discuss the global time-space maximal estimate
\begin{equation}\label{global-global0}
	\big\| \sup_{t \in \mathbb R} \big|e^{it\sqrt{-\Delta}}f\big|\big\|_{L^{q}(\mathbb R^d)} \le C_{n,q,s} \|f\|_{\dot{H}^s(\mathbb{R}^d)},
\end{equation}
where $\dot{H}^s(\mathbb{R}^d)$ denotes the homogeneous Sobolev space of order $s$. As can be easily seen by a scaling argument,  for the global estimate \eqref{global-global0}  we need to use the homogeneous Sobolev space $\dot{H}^s$. Moreover,   the scaling condition $s = \frac{d}2-\frac{d}q$ has to be satisfied. 

\begin{proposition}
\label{prop:local-global to global}
Let $q\ge q_\circ(d)$ and $s= s_c(q,d)= \frac{d}{2}-\frac{d}{q}$. Then, the estimates \eqref{local-local} and \eqref{global-global0} are equivalent.
\end{proposition}

It is easy to see that the equivalence in Proposition \ref{prop:local-global to global} readily generalizes to the Lorentz spaces $L^{q,r}$ by replacing $L^q$ (see {\bf (G4)} in {\it Generalizations} below and Section \ref{proof-scaling} below). Furthermore, as to be seen later, by real interpolation the estimate  \eqref{max-c}  extends to the ${H}^{s_c(q,d)}$--$L^{q,2}$ estimate (see Remark \ref{lolo}).  
Consequently, we have the following. 

\begin{theorem}\label{cor:local-global}
Let $q\in (q_\circ(d),\infty)$. Then \eqref{global-global0} holds with $s=s_c(q,d)$. Moreover, we have
\begin{equation*}\label{global-globall}
	\big\| \sup_{t \in \mathbb R} \big|e^{it\sqrt{-\Delta}}f\big|\big\|_{L^{q,2}(\mathbb R^d)} \le C_{n,q,s}  \|f\|_{\dot H^{s_c(q,d)}(\mathbb R^d)}. 
\end{equation*}
\end{theorem}

Theorem \ref{cor:local-global} generalizes the result in \cite{BG12}, where only the case \( d = 3 \) and \( q \in [6, \infty) \) was considered.
We close the introduction with some remarks concerning extensions of our results to more general evolution operators.

\subsubsection*{Generalizations.} 
\label{general} For a function $\phi$ that is smooth on $\mathbb R^d\setminus \{0\}$, let  us consider
\[ e^{it\phi(D)} f= (2\pi)^{-d} \int  e^{i(x\cdot\xi +t\phi(\xi))} \widehat f(\xi)\, d\xi.  \] 
In what follows we list some of the possible generalizations, which can be deduced from the arguments in this paper without difficulty. Though further extensions are  possible, we do not attempt  to present them in most general forms.

\smallskip

\noindent {\bf(G1)}  If $\phi$ is a smooth function that is  homogeneous of degree $1$,  and the Hessian matrix of $\phi$ has $d-1$ nonzero eigenvalues of the same sign, then the endpoint estimates in Theorem \ref{thm:max} remain valid for the operator $e^{it\phi(D)}$ replacing the wave operator $e^{it\sqrt{-\Delta}}$. This holds since the same bilinear restriction estimates hold for the surface $(\xi, \phi(\xi))$ (see  Theorem \ref{thm:bi} below). 

\smallskip

\noindent {\bf(G2)} It is natural to expect  that the same endpoint estimates continue to hold for the operator $e^{it\phi(D)}$  whose dispersive symbol $\phi$  is no longer homogeneous, such as the   Klein--Gordon equation.  Indeed, suppose 
that $\phi$ is bounded on every bounded set, and suppose that there are constants $C, R>0$ such that
\Be 
\label{g2g2}
|\partial^\alpha (\phi(\xi) -|\xi|)|\le C|\xi|^{-|\alpha|}
\Ee
for $|\alpha|\le d+1$ if $|\xi|\ge R$. Then, the  estimates in Theorem \ref{thm:max} hold true  for the operator $e^{it\phi(D)}.$  
Typical examples include the operators $e^{it\phi(D)}$ given by  $\phi(\xi)=(1+ |\xi|^\kappa)^{1/\kappa}$ with $\kappa\ge 1$.  
See Section \ref{remark:b} for further details.

\smallskip

\noindent {\bf(G3)}  Proposition \ref{prop:xlocal to global} also remains valid even if  the wave operator $e^{it\sqrt{-\Delta}}$ is replaced by $e^{it\phi(D)}$ provided that there are  constants  $C, R>0$ such that
\Be\label{g3}  |\partial^\alpha \phi(\xi)|\le  C|\xi|^{1-|\alpha|}  \Ee
for $|\xi|\ge R$  if $|\alpha|\le d+1$. See Remark \ref{rmk:g3g3} below.
 
\smallskip

\noindent {\bf(G4)} The implications between $H^s$ (or $\dot{H}^s$)--$L_x^q L_t^\infty$ estimates can readily be extended to $H^s$ (or $\dot{H}^s$)--$L_x^q L_t^r$ estimates under appropriate conditions on $s$, $q$, and $r$. The details are left to the interested reader.

\subsubsection*{Organization of the paper}
This paper is organized as follows: Section 2 is devoted to obtaining the optimal $L^p$--$L^q$ smoothing estimates for the wave operator with 
frequency localization, which serve as the key ingredients for the proof of Theorem \ref{thm:max}. In Section 3, we provide the proofs of Theorems \ref{thm:max} and \ref{cor:local-global}. Finally, we prove the equivalences of  various estimates, as stated in Propositions \ref{prop:xlocal to global}, \ref{prop:local to global}, and \ref{prop:local-global to global}.

\section{$L^p-L^q$ estimates for the wave operator}

In order to show Theorem \ref{thm:max}, we make use of $L^p$--$L^q$ type local smoothing estimate for the wave operator. 
In order to state our result,  we first introduce some notation. 

Let us set $\mathfrak Q$ be a point in the unit square $[0,1]^2$ given by
\[   \mathfrak Q=\Big( \frac{(d-1)(d+3)}{2(d^2+2d-1)} , \frac{(d-1)(d+1)}{2(d^2+2d-1)} \Big),\]
on which  the lines $L_1:(d-1)(1-x)=(d+1) y$ and $L_2: y=\frac{d+1}{d+3} x$ intersect each other.  We denote 
by $\mathcal T_u$  the closed triangle with vertices $(1,0), (1/2,1/2)$, and $ \mathfrak Q$. By $\mathcal T_l$
we also denote 
 the closed triangle with vertices $(1,0), (0,0)$, and $ \mathfrak Q$.  
 (See Figure \ref{fig2} 
  below.) 
 For $1\le p, q\le \infty$,  let us 
 set 
  \begin{align}
 \label{def-betau}
 \beta_u(p, q)&= \tfrac {d+1}{2}\big(\tfrac1p-\tfrac1{q}\big), 
 \\[4pt]
  \label{def-betal}
 \beta_l(p, q)&= \tfrac{d-1}2+\tfrac{1}{p}- \tfrac{d+1}q. 
 \end{align} 
We also define 
  \[
  \beta(p, q)=
 \begin{cases}  
\beta_u(p, q), & (\frac1p, \frac1q)\in  \mathcal T_u,
 \\[4pt]
\beta_l(p, q), & (\frac1p, \frac1q)\in  \mathcal T_l.
 \end{cases}
 \] 
Note that $\beta_u(p, q)=\beta_l(p, q)$ if $(\frac1p, \frac1q)\in \mathcal T_u\cap \mathcal T_l.$ So, $\beta(p, q)$ is well-defined.  

\begin{proposition} 
\label{prop:wave}
Let $I_\circ= (1,2)$  and $N\ge 1$. Let $(\frac1p, \frac1q)\in \mathcal T_u\cup \mathcal T_l\setminus{\{ \mathfrak Q\}}$. 
 Suppose that $\supp \widehat f$ is included in $\mathbb A_N=\{ \xi: N/2\le  |\xi|\le 2N\}$. Then,  
\Be
\label{smoothing}
\| e^{it\sqrt{-\Delta}} f\|_{L^{q}( \mathbb R^d\times I_\circ)}  \le CN^{\beta(p,q)} \|f\|_{L^p}.
\Ee
\end{proposition} 

It should be noted that the bounds in \eqref{smoothing} are no longer valid if the interval $I_\circ$ is replaced by $(0,1)$. 
This is why we use $I_\circ$ instead of the interval $(0,1)$.
Sharpness of the estimate \eqref{smoothing} can be shown by adapting the examples in \cite{TV2}.

The estimates \eqref{smoothing}  for the cases $(p, q)=(\infty,\infty), (1,\infty),$ and $(2, 2)$ are easy to show  (well-known). In fact, \eqref{smoothing}  for $(p, q)= (2, 2)$ follows from Plancherel's theorem. For the other cases, see Section \ref{001} below.  Consequently, the proof of Proposition \ref{prop:wave} basically  reduces to proving the estimate   
\Be 
\label{smoothing0}
\| e^{it\sqrt{-\Delta}} f\|_{L^{q,\infty}( \mathbb R^d\times I_\circ)}  \le CN^{\beta(p,q)} \|f\|_{L^{p, 1}}, \quad   (\frac1p, \frac1q)= \mathfrak Q.
\Ee 
Once we obtain 
\eqref{smoothing0}, 
by  real interpolation between this and the above mentioned estimates,  the estimate \eqref{smoothing}  follows for all $(\frac1p, \frac1q)\in \mathcal T_u\cup \mathcal T_l\setminus{\{ \mathfrak Q\}}$. 

 We show the estimate \eqref{smoothing0} by making use of the bilinear adjoint restriction estimate to the cone. The proof of \eqref{smoothing0}  is to be provided at the end of this section.    

\subsection{Bilinear estimates for the wave operator}
Let  $\mathcal T$ denote the closed triangle with vertices $(0,0),  \mathfrak P:=(\frac 12, \frac{d+1}{2(d+3)}),$ and $ (1,0)$. 
The following is a consequence of the bilinear adjoint restriction estimate to the cone, which is due to 
Wolff \cite{W} and Tao \cite{T}.

\begin{figure}[t]
\begin{centering}
\begin{tikzpicture}[scale=7]
    \draw[->] (0,0) -- (1.2,0) node[below] {\(\frac1p\)};
    \draw[->] (0,0) -- (0,0.7) node[left] {\(\frac1q\)};
    
    \fill[violet!50,opacity=0.6] (0,0) -- (1,0) -- (5/14,3/14) -- cycle; 
    \node[font=\small] at (0.4,0.07) {\(\mathcal{T}_l\)};
    
    \fill[LimeGreen!80,opacity=0.6] (1,0) -- (0.5,0.5) -- (5/14,3/14) -- cycle; 
    \node[font=\small] at (0.65,0.2) {\(\mathcal{T}_u\)};
    
    \draw[thick] (5/14,3/14) -- (1,0);
    
        \draw[densely dotted] (0.5,0.3) -- (5/14,3/14);
    
    \draw[densely dotted] (0,0) -- (0.7,0.7);
    \draw[thick] (5/14,3/14) -- (0.5,0.5);
    \draw[thick] (1,0) -- (0.5,0.5);

    \draw[red, thick] (0.5,0) -- (0.5,0.5);
    
    \filldraw[black] (0.5,0.3) circle (0.2pt) node[font=\footnotesize, right] {\( \mathfrak P\)};
    \filldraw[black] (5/14,3/14) circle (0.2pt) node[font=\footnotesize, below] {\( \mathfrak Q\)};
    
    \node[below left] at (0,0) {\((0,0)\)};
    \node[left] at (0,0.5) {\(\frac12\)};
    \node[below] at (0.5,0) {\(\frac12\)};
    \node[below] at (1,0) {\(1\)};
    \node[left] at (0,1/6) {\(\frac 1{q_\circ(d)}\)};
    
    \draw[thick] (5/14,3/14) -- (0,0) node[font=\footnotesize, pos=0.3, left] {\(\!L_2\,\,\)};
    \draw[draw=black] (0.25,0.25) -- (5/14,3/14) node[font=\footnotesize, pos=0.7,above] {\(L_1\)};
\end{tikzpicture}
\caption{The range of $(1/p, 1/q)$ for the estimate \eqref{smoothing}.} \label{fig2}
\end{centering}
\end{figure}
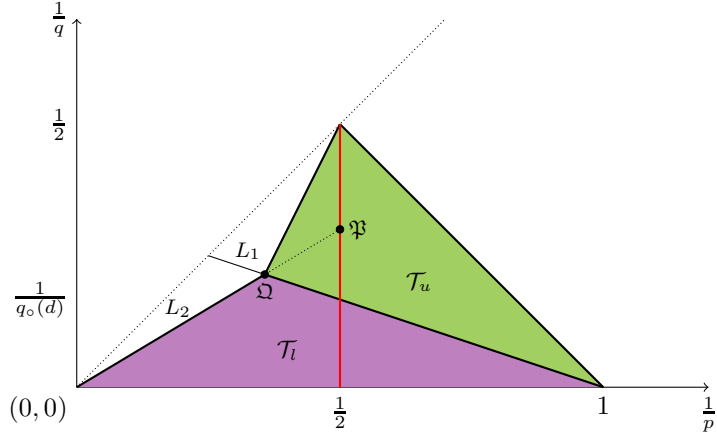

\begin{theorem}
\label{thm:bi}  Let $N\ge 1$ and  $1/\sqrt N \le \theta \le 1/10$.   Let $\Theta, \Theta'\subset \mathbb S^{d-1}$ be spherical caps of diameter $\theta$. Suppose  that 
\begin{align*}
 \supp \widehat f\subset \Lambda&:=\{ \xi:  \xi/|\xi|\in \Theta, \ N/2\le  |\xi|\le 2N \},   
  \\
   \supp \widehat g\subset \Lambda'&:=\{ \xi:  \xi/|\xi|\in \Theta', \ N/2\le  |\xi|\le 2N \}.
  \end{align*}
   Additionally, suppose $\dist(\Theta, \Theta')\ge  \theta$.  Then, for $p, q$ satisfying  $(\frac 1p, \frac1q)\in \mathcal T$, we have 
\Be 
\label{bilinear}
\| e^{it\sqrt{-\Delta}} f  e^{it\sqrt{-\Delta}} g \|_{L^{\frac q2}( \mathbb R^d\times  I_\circ)} \lesssim  
N^{2\beta_l(p,q)}
 \theta^{ 4(\beta_l(p,q)-\beta_u(p,q))} \|f\|_{L^p} \|g\|_{L^p}. 
\Ee
Here $\beta_l(p,q), \beta_u(p,q)$ are given by \eqref{def-betau} and \eqref{def-betal}. 
 Furthermore, if $\theta\sim 1/\sqrt N$, then \eqref{bilinear} holds without the assumption $\dist(\Theta, \Theta')\sim \theta$. 
\end{theorem}

When $\theta\sim 1$, $N\sim 1$, and $p=2$,  the estimate \eqref{bilinear} is known as the bilinear estimate for the cone. The estimate   \eqref{bilinear}  for $p=2$ and  $q>\frac{2(d+3)}{d+1}$  is due to  Wolff \cite{W} and the endpoint case for $p=2$ and $q=\frac{2(d+3)}{d+1}$ was obtained  by Tao \cite{T}.

In order to show Theorem \ref{thm:bi}, it is sufficient to verify 
\eqref{bilinear} for the cases $(p,q)=(1,\infty)$, $(p,q)=(\infty,\infty)$, and 
$(p,q)=(2,\frac{2(d+3)}{d+1})$.  For the last case we need to use the bilinear restriction estimates while the other two cases can be shown by kernel estimates, 
which are rather straightforward.  

 By rotation we may assume  that the caps $\Theta, \Theta'$ are contained in a $C\theta$ neighborhood of $e_1$. More precisely, we may assume 
 \[   \supp \widehat f,\ \supp \widehat g\subset \Lambda_0=\{ \xi:   \xi_1\in [N/2, 2N], \  |\xi'| \le  C \theta N    \},   \]
where $\xi=(\xi_1, \xi')\in\mathbb R\times \mathbb R^{d-1}$. Since $\dist(\Theta, \Theta')\ge  \theta$, by additional rotation in $\xi'$ and harmless dilation 
we may assume  that
\Be
\label{sep}
\begin{aligned}
  \supp \widehat f &\subset \{ \xi\in \Lambda_0:  |\xi'| \ge   2c \theta N    \},   
  \\
    \supp \widehat g&\subset \{ \xi\in \Lambda_0:  |\xi'| \le   c \theta N    \}
  \end{aligned}
  \Ee
for a constant $c>0$. 

\subsubsection{\bf Proof of \eqref{bilinear} for $(p,q)=(\infty, \infty)$ and $(p,q)=(1, \infty)$} \label{001}
Let $\beta_0\in C_c^\infty(-2^2, 2^2)$ such that  $\beta_0=1$ on $[-2, 2]$, and let $\beta \in C_c^{\infty}(2^{-2}, 2^2)$ such that $\beta=1$ on $[2^{-1}, 2]$. 

Let us set 
\[K_t^N(x)=(2 \pi)^{-d} \int  e^{i (x\cdot \xi+t|\xi|)} \beta(|\xi|/N) d\xi. \] 
Note that $K_t^N(x)=(2 \pi)^{-d} N^d \int  e^{iN (x\cdot \xi+t|\xi|)} \beta(|\xi|) d\xi$. 
Since the Hessian matrix of the function $\xi\to |\xi|$ has rank $d-1$ on the support of $\beta$, by the stationary phase method we see that $\|K_t^N\|_\infty\lesssim N^\frac{d+1}2$.   
Let us set 
\[  \chi_\theta (\xi)= \beta\Big(\frac{\xi_1}{N}\Big)\beta_0\Big(\frac {|\xi'|}{c\theta N}\Big) . \] 
Note that $e^{it\sqrt{-\Delta}} f=K_t^N\ast \chi^\vee_\theta\ast f$ and $\|K_t^N\ast \chi^\vee_\theta\|_{L^\infty} \le \|K_t^N\|_{L^\infty} \|\chi^\vee_\theta\|_{L^1} \lesssim  N^\frac{d+1}2$. Thus,  \eqref{bilinear} follows for $(p,q)=(1, \infty)$.

It is not difficult to see that  
$\|K_t^N\ast \chi^\vee_\theta\|_{L^1}\le C$ 
if $\theta\le N^{-1/2}$. Thus, dividing the support of 
$\chi_\theta$ into as many as $(\theta N^\frac12)^{d-1}$   sectors of  angular diameter $N^{-\frac12}$, we have $\|K_t^N\ast \chi^\vee_\theta\|_1\le C(\theta N^\frac12)^{d-1}$ for $t\in  I_\circ$. Thus, we get \eqref{bilinear} for $(p,q)=(\infty, \infty)$. 

\subsubsection{\bf Proof of \eqref{bilinear} for $(p,q)=(2,\frac{2(d+3)}{d+1})$}  Finally, to show 
\eqref{bilinear} for $(p,q)=(2,\frac{2(d+3)}{d+1})$, we use the bilinear endpoint restriction estimate to the cone due to Tao \cite{T}.  

 By  scaling  (i.e., $\widehat f  \to   \widehat f (N\cdot)$  and $\widehat g  \to   \widehat g (N\cdot)$), we may assume 
$N=1$. Consequently, it is enough to show 
\Be \label{reduced} \| e^{it\sqrt{-\Delta}} f  e^{it\sqrt{-\Delta}} g \|_{L^{r}( \mathbb R^d\times \mathbb R)} \lesssim    \theta^{d-1-\frac{d+1}r} \|f\|_{L^2} \|g\|_{L^2}\Ee
with $r=\frac {d+3}{d+1}$  while \eqref{sep}  holds. 

 Let us set
\[ \mathcal E h(x,t)= \iint e^{i(x_1\rho+ x'\cdot \eta+ t  |\eta|^2/\rho)}  h(\rho, \eta) d\eta d\rho. \]  
By changing variables $(x,t)\to (x_1-t, x', x_1+t)$,  the phase function $(x,t)\cdot (\xi, |\xi|)$ for $e^{it\sqrt{-\Delta}} f$ is transformed to $(x,t)\cdot(\xi_1+|\xi|, \xi',  |\xi|-\xi_1)$. By an additional change of variables $(\xi_1+|\xi|, \xi')\to (\rho, \eta)$,  we note that $|\xi|-\xi_1=|\eta|^2/\rho$, thus by the Plancherel's theorem and discarding harmless factors resulting from the change of variables,  the estimate \eqref{reduced} is now equivalent to the estimate 
\Be 
\label{theta} 
\| \mathcal E h_1  \mathcal E h_2 \|_{L^{r}( \mathbb R^d\times \mathbb R)} \lesssim    \theta^{d-1-\frac{d+1}r} \|h_1\|_{L^2} \|h_2\|_{L^2},
\Ee
while 
\begin{align*}
  \supp h_1 \subset  \mathfrak A_\theta&:= \{ (\eta, \rho) : \rho\sim 1,   |\eta| \ge   2c \theta \}, 
  \\
     \supp h_2\subset \mathfrak A_\theta'&:= \{ (\eta, \rho) : \rho\sim 1,  |\eta| \le   c \theta   \}
\end{align*} 
for a constant $c>0$. 
Note that \eqref{theta} with $\theta\sim 1$ is equivalent to the bilinear restriction estimate for the cone. 

Now, we set  
\[\tilde h_j(\rho, \eta)= \theta^{\frac{d-1}2}h_j(\rho, \theta \eta), \quad j=1,2.\]  
Then, by the change of variables $\eta\to \theta \eta$  we see 
\[ (\mathcal E h_1  \mathcal E h_2)(x,t)=  \theta^{(d-1)}(\mathcal E \tilde h_1  \mathcal E \tilde h_2)(x_1, \theta x',\theta^2 t).   \]
Since $\supp  \tilde h_1\subset \mathfrak A_1$ and $\supp  \tilde h_2\subset \mathfrak A_1'$, we may apply the estimate \eqref{theta} with $\theta=1$, which is in fact the bilinear restriction estimate to the cone. Therefore, the desired estimate 
 \eqref{theta} follows.

For the case $\theta\sim 1/\sqrt N$, we do not need the condition that $\dist(\Theta, \Theta')\ge  \theta$. 
To show \eqref{bilinear} for the case $\theta\sim 1/\sqrt N$,  by H\"older's inequality  it is enough to show 
\[
\| e^{it\sqrt{-\Delta}} f \|_{L^{q}( \mathbb R^d\times  I_\circ)} \lesssim  N^{\frac{d+1}2(\frac1p-\frac{1}q)}   \|f\|_{L^p}
\]
for $2\le p\le q\le \infty$. In fact, the estimate for $(p,q)=(2,2)$  follows from the Plancherel's theorem. 
The cases $(p,q)=(\infty, \infty),$ $(1,\infty)$ are clear from the previous computations in Section \ref{001}. Finally, interpolation gives all the  desired estimates.

\newcommand{\diam}{\operatorname{diam}}

\subsection{Proof of  \eqref{smoothing0}}
Before we prove the inequality \eqref{smoothing0}, firstly we recall a couple of lemmas. The following  is a multilinear generalization of a lemma known as Bourgain's summation trick \cite{B}.

\begin{lemma}[{\cite[Lemma 2.6]{L}}]
\label{B-Sum-Tri} Let $n\ge 1$. 
Let  $1\leq p_0^k,p_1^k\le \infty$, $k=1,\cdots, n,$ and  $1\leq q_0, q_1\le\infty$. Suppose that $\{T_j\}_{j=-\infty}^\infty$ is a sequence of $n$-linear (or sublinear) operators such that
\begin{equation*}
\|T_j(f^1,\cdots,f^n)\|_{L^{q_\ell}}\leq M_l 2^{j (-1)^\ell \varepsilon_\ell }\prod_{k=1}^n\|f^k\|_{L^{p_\ell^k}}, \quad \ell=0, 1
\end{equation*}
for some  $\varepsilon_0$, $\varepsilon_1>0$. 
Then, we have 
\begin{equation*}
\|\sum_jT_j(f^1,\cdots,f^n)\|_{L^{q,\infty}}\leq CM_1^{\theta}M_2^{1-\theta}\prod_{k=1}^{n}\|f^k\|_{L^{p^k,1}},
\end{equation*}
where $\theta=\varepsilon_1/(\varepsilon_0+\varepsilon_1)$, $1/q=\theta/q_0+(1-\theta)/q_1$, and $1/p^k=\theta/p_0^k+(1-\theta)/p_1^k$.
\end{lemma}

We also recall the following useful lemma from \cite{TV2} to make use of orthogonality between the decomposed operators. 
\begin{lemma}[{\cite[Lemma 7.1]{TV2}}]\label{TVV-lemma}
Let $\{R_k\}$ be a collection of rectangles in frequency space such that the dilates $\{2R_k\}$ are essentially disjoint. Suppose that $\{F_k\}$ is a collection of functions whose Fourier transforms are supported in $R_k$. Then, for $1\leq p\leq \infty$, we have
\begin{equation*}
\biggl(\sum_k\|F_k\|_{L^p}^{\max(p,p')}\biggr)^{1/\max(p,p')}\lesssim \|\sum_kF_k\|_{L^p}\lesssim \biggl(\sum_k\|F_k\|_{L^p}^{\min(p,p')}\biggr)^{1/\min(p,p')}.
\end{equation*}
\end{lemma}

By finite decomposition and rotation,  we may assume  
\[ \supp \widehat f\subset  \mathbb A_N^0:= \{ \xi:  \xi/|\xi|\in \Theta_0, \ N/2\le  |\xi|\le 2N \},\]
where $\Theta_0$ is a small spherical cap around  $e_1$. 
We decompose  $\Theta_0\times {\Theta_0}$ using a Whitney type decomposition of    $\Theta_0\times {\Theta_0}$ away from its diagonal. 

\subsubsection*{A Whitney type decomposition.}  Following the typical dyadic decomposition process, for each $\nu\ge 0$, we partition
$\Theta_0$ into spherical caps $\Theta_k^\nu$  such that  
\[c_d  2^{-\nu}\le 
  \diam (\Theta_k^\nu)  \le C_d2^{-\nu}\] for some constants $c_d$, $C_d>0$
and 
$\Theta_k^\nu \subset \Theta_{k'}^{\nu'}$ for some $k'$ whenever $\nu\ge \nu'$. Let  $\nu_\circ=\nu_\circ(N)$ denote 
 the integer  $\nu_\circ$ such that 
\[  N^{-1} <   2^{-2\nu_\circ} \le  4N^{-1} .  
\]
Thus, $\Theta_0=\cup_{k}\Theta_k^\nu$ for each $\nu$.  Consequently,  we may write 
\Be
\label{bi-decomp} \Theta_0\times \Theta_0=\bigcup_{ \nu:  2^{-\nu_\circ} \le    2^{-\nu} \lesssim 1 }\,\, \bigcup_{k\sim_\nu k'} \Theta_k^\nu\times \Theta_{k'}^{\nu},\Ee
where  $k\sim_\nu k'$ means $\dist (\Theta_k^{\nu}, \Theta_{k'}^{\nu})\sim 2^{-\nu}$ if  $\nu< \nu_\circ$ and  $\dist (\Theta_k^{\nu}, \Theta_{k'}^{\nu})\lesssim 2^{-\nu}$ if  $\nu= \nu_\circ $ 
(e.g., see \cite[p.971]{TVV}). When  $k\sim_{\nu_\circ} k'$, the sets $\Theta_k^{\nu_\circ}$ and $\Theta_{k'}^{\nu_\circ}$ are not necessarily separated from each other by 
distance $\sim 2^{-\nu_\circ}$ since  the decomposition 
process terminates  at $\nu=\nu_\circ$.

Instead of \eqref{smoothing0}, we consider the bilinear estimate for $(f,g)\mapsto e^{it\sqrt{-\Delta}} f  e^{it\sqrt{-\Delta}} g$ from 
$L^p\times L^p$ to $L^{q/2}$, which clearly implies \eqref{smoothing0}. 
We assume that 
\[   \supp \widehat f, \quad  \supp \widehat g\, \subset  \mathbb A_N^0.\] 
Let us define $f_k^\nu$ and $g_{k'}^\nu$ by
\[  \mathcal F (f_k^\nu)( \xi)=    \widehat f(\xi) \chi_{\Theta_k^\nu} \Big(\frac{\xi}{|\xi|}\Big), \quad    \mathcal F (g_{k'}^\nu)( \xi)=    \widehat g(\xi) \chi_{\Theta_{k'}^\nu} \Big(\frac{\xi}{|\xi|}\Big)     .\] 
Since, for each $\nu$, $\Theta_0=\cup_{k}\Theta_k^\nu$ and the sets  
$\Theta_k^\nu$ are essentially disjoint,  $\sum_{k}  \chi_{\Theta_k^\nu}=1$ almost everywhere.  Thus, it is clear that 
\[ \widehat f=\sum_k  \mathcal F (f_k^\nu), \quad \widehat  g=\sum_k  \mathcal F (g_k^\nu).\]
Combining this and  \eqref{bi-decomp}, we may write 
\Be 
\label{bi-decomp2} \widehat f(\xi) \widehat g(\eta)=  \sum_{\nu_\circ\ge \nu}  \sum_{k\sim_\nu k'}   \mathcal F (f_k^\nu)(\xi)\,\mathcal F (g_{k'}^\nu)(\eta).\Ee

Now, for each $\nu$,  define a bilinear operator 
\[  B^\nu (f,g) (x,t)=   \sum_{k\sim_\nu k'} e^{it\sqrt{-\Delta}}  f_k^\nu(x)\, e^{it\sqrt{-\Delta}} g_{k'}^\nu(x).   \] 
From \eqref{bi-decomp2}, it follows that 
\[   e^{it\sqrt{-\Delta}} f  e^{it\sqrt{-\Delta}} g =\sum_{\nu_\circ\ge \nu}    B^\nu (f,g).\]

Using orthogonality (Lemma \ref{TVV-lemma})  and Theorem \ref{thm:bi},  for  $\nu\le \nu_0$, we obtain 
\[   \| B^\nu (f,g)\|_{L^{\frac q2}( \mathbb R^d\times  I_\circ)}\lesssim    N^{2\beta_l(p,q)}  2^{ -4\nu(\beta_l(p,q)-\beta_u(p,q))} \|f\|_{L^p} \|g\|_{L^p} \] 
provided that  $(1/p, 1/q)\in  \mathcal T_0:=\mathcal T\cap\{(a,b): 1/2\le a+b, \, 1/4\le a\le 1/2\}$. 
Indeed, using Lemma \ref{TVV-lemma}, we have 
\Be
\label{haha}   \| B^\nu (f,g)\|_{L^{\frac q2}( \mathbb R^d\times  I_\circ)} 
\lesssim  \Big( \sum_{k\sim_\nu k'} \|e^{it\sqrt{-\Delta}}  f_k^\nu e^{it\sqrt{-\Delta}} g_{k'}^\nu  \|_{L^{\frac q2}( \mathbb R^d\times  I_\circ)}^{q_\star} \Big)^\frac{1}{q_\star} 
\Ee
where $q_\star=\min(q/2, (q/2)')$. To show this, note that, for each fixed $t$,   the Fourier transforms of  $\{ e^{it\sqrt{-\Delta}}  f_k^\nu e^{it\sqrt{-\Delta}} g_{k'}^\nu\}_{k\sim_\nu k'} $ 
are supported in boundedly overlapping rectangle of dimension about $N\times \overset{(d-1)-\text{times}}{ N2^{-\nu}\times\dots \times N2^{-\nu}}$. Hence, 
Lemma \ref{TVV-lemma} (the second inequality) gives
\[
\| B^\nu (f,g)(\cdot, t)\|_{L^{\frac q2}( \mathbb R^d)} 
\lesssim  \Big( \sum_{k\sim_\nu k'} \|e^{it\sqrt{-\Delta}}  f_k^\nu e^{it\sqrt{-\Delta}} g_{k'}^\nu  \|_{L^{\frac q2}( \mathbb R^d)}^{q_\star} \Big)^\frac{1}{q_\star} .\]
Since $q/2 \ge q_\star$, raising power to $q/2$ and taking integration over $I_\circ$ give \eqref{haha} via Minkowski's inequality. 
 By combining \eqref{haha} and  Theorem \ref{thm:bi}  and   Cauchy-Schwarz inequality,  we see
\begin{align*}
\|B^\nu (f,&g)\|_{L^{\frac q2}( \mathbb R^d\times  I_\circ)}\\
 &\lesssim  N^{2\beta_l(p,q)}  2^{ -4\nu(\beta_l(p,q)-\beta_u(p,q))}   \Big( \sum_{k\sim_\nu k'}  \|f_k^\nu\|_{L^p}^{q_\star} \|g_{k'}^\nu\|_{L^p}^{q_\star} \Big)^\frac{1}{q_\star}  
\\ 
&\le   N^{2\beta_l(p,q)}  2^{ -4\nu(\beta_l(p,q)-\beta_u(p,q))}  \Big( \sum_{k}  \|f_k^\nu\|_p^{2q_\star}\Big)^\frac{1}{2q_\star}
 \Big( \sum_{k'}   \|g_{k'}^\nu\|_{L^p}^{2q_\star}\Big)^\frac{1}{2q_\star}. 
\end{align*}

Note that  $2q_\star\ge p\ge 2$ whenever $(1/p, 1/q)\in \mathcal T_0$.  
 Therefore,  since $f=\sum_k f_k^\nu$ and $g=\sum_k g_k^\nu$ for each $\nu$, Lemma \ref{TVV-lemma} (the first inequality) gives
$\sum_{k}  \|f_k^\nu\|_{L^p}^{2q_\star}\lesssim \|f\|_{L^p}^p$ and $ \sum_{k'}   \|g_{k'}^\nu\|_{L^p}^{2q_\star}\lesssim \|g\|_{L^p}^p$ for $(1/p, 1/q)\in \mathcal T_0$. 
Consequently, we obtain the  estimate
\Be
\label{the-estimate} \| B^\nu (f,g)\|_{L^{\frac q2}( \mathbb R^d\times  I_\circ)}\lesssim N^{2\beta_l(p,q)}  2^{ -4\nu(\beta_l(p,q)-\beta_u(p,q))} \|f\|_{L^p} \|g\|_{L^p}
\Ee
 for $(1/p, 1/q)\in \mathcal T_0$.

Note that 
\[4(\beta_l(p,q)-\beta_u(p,q))= 2 \big((d-1)(1-\tfrac1p)-\tfrac{d+1}q\big).\]
   Recall that the lines $L_1: (d-1)(1-x)=(d+1) y$ and $L_2: y=\frac{d+1}{d+3} x$ intersect at $(1/{p}_\ast,1/{q}_\ast):= \mathfrak Q$. Also note that $L_2$ contains the open line segment $((0,0),  \mathfrak P)\cap \mathcal T_0$, which 
contains  the point $ \mathfrak Q$. (See Figure  \ref{fig2}.) We choose $(1/p_0, 1/q_0)$, $(1/p_1, 1/q_1)\in [(0,0), \mathfrak P]\cap \mathcal T_0$  such that 
\[\beta_l(p_0,q_0)-\beta_u(p_0,q_0)> 0> \beta_l(p_1,q_1)-\beta_u(p_1,q_1).\] 
Consequently, we have the estimate \eqref{the-estimate} for $(p,q)=(p_0,q_0)$ and $(p,q)=(p_1,q_1)$. 
Note that  $\beta_l(p_\ast,q_\ast)=\beta_u(p_\ast,q_\ast)$.   We apply  Lemma \ref{B-Sum-Tri} with $n=2$  and those two estimates to obtain 
 the restricted weak type bound
\[ \|e^{it\sqrt{-\Delta}} f  e^{it\sqrt{-\Delta}} g\|_{L^{{q}_\ast/2,\infty}(\mathbb R^d\times I_\circ)}\lesssim   N^{2\beta_l(p_\ast,q_\ast)} \|f\|_{L^{{p}_\ast, 1}} \|g\|_{L^{{p}_\ast, 1}}, \]
which obviously gives the desired estimate for \eqref{smoothing0}.  Hence, this completes the proof of \eqref{smoothing0}.

\section{Proof of Theorem \ref{thm:max}}
 In this section we prove Theorem \ref{thm:max}. 
 Once we have Proposition \ref{prop:wave}, the proof of Theorem \ref{thm:max} is rather standard.  
 
In order to show \eqref{max}, by time translation and Plancherel's theorem we  may replace the interval $(0,1)$ with $I_\circ=(1,2)$. 
More precisely, \eqref{max} is equivalent with 
\begin{equation}\label{max1}
	\big\| \sup_{t \in I_\circ} \big|e^{it\sqrt{-\Delta}}f\big|\big\|_{L^{q}(\mathbb{R}^d)} \le C_{d,q,s} \|f\|_{{H}^s(\mathbb{R}^d)}. 
\end{equation}
We prove \eqref{max1} by considering the cases $q\in (2, q_\circ(d))$ and $q\in (q_\circ(d), \infty)$, separately.

We begin with  the maximal estimate 
\begin{equation}\label{max-frequency}
\| \sup_{t\in I_\circ}  |e^{it\sqrt{-\Delta}} f|\|_{L^{q}( \mathbb R^d)}  \le CN^{\beta(p,q)+\frac1q} \|f\|_{L^p},
\end{equation}
for $(\frac1p, \frac1q)\in \mathcal T_u\cup \mathcal T_l\setminus{\{ \mathfrak Q\}}$ provided that $\supp \widehat f \subset \mathbb A_N$.  Although the estimate \eqref{max-frequency} follows from Proposition \ref{prop:wave} and a standard argument, we provide some details for the reader's convenience. 

 From \cite[Lemma 1.3]{L} we recall that 
\[
\big\| \sup_{t\in I_\circ}|F(x,t)| \big\|_{L^q(\mathbb{R}^d)} 
\leq C \Big( \|F\|_{L^q(\mathbb{R}^d \times I_\circ)} 
+ \|F\|_{L^q(\mathbb{R}^d \times I_\circ)}^{\frac{q-1}{q}} 
\|\partial_t F\|_{L^q(\mathbb{R}^d \times I_\circ)}^{\frac{1}{q}} \Big), 
\]
which is an easy consequence of the fundamental theorem of calculus and H\"older's inequality.  By Young's  inequality, 
we have 
\[ \|F\|_{L^q(\mathbb{R}^d \times I_\circ)}^{\frac{q-1}{q}} 
\|\partial_t F\|_{L^q(\mathbb{R}^d \times I_\circ)}^{\frac{1}{q}} \le C \Big( N^\frac1q \|F\|_{L^q(\mathbb{R}^d \times I_\circ)}  +  N^\frac{1-q}q \|\partial_t F\|_{L^q(\mathbb{R}^d \times I_\circ)} \Big). \] 
Thus, applying the above inequality with $F(x,t) =e^{it\sqrt{-\Delta}} f(x)$, we have  the left-hand side of \eqref{max-frequency} bounded by a constant multiple of  
\[ 
 N^\frac1q \|e^{it\sqrt{-\Delta}}f\|_{L^q(\mathbb{R}^d \times I_\circ)}  +  N^\frac{1-q}q \| e^{it\sqrt{-\Delta}} (\sqrt{-\Delta} f)\|_{L^q(\mathbb{R}^d \times I_\circ)}.   \]
 Since $\supp \widehat f \subset \mathbb A_N$,  Mikhlin's multiplier theorem gives $\|\sqrt{-\Delta} f\|_p\le C N\|f\|_p$. Therefore,   we obtain  \eqref{max-frequency} by Proposition \ref{prop:wave}. 

\subsection{The case $q\in (2, q_\circ(d))$}  We first show \eqref{max} for the case $q\in (2, q_\circ(d))$. Indeed, fixing  $q_\ast \in (2, q_\circ(d))$, 
 we show 
 \begin{equation*}
 \label{max00} \Big \| \sup_{t\in I_\circ} |e^{it\sqrt{-\Delta}} f|\Big\|_{L^{q_\ast}(\mathbb R^d)} \le C\|f\|_{H^{s_c(q_\ast)}},  
\end{equation*}
where 
\[ s_\ast=s_c(q_\ast, d)=(d+1)/4-(d-1)/2q_\ast.\]

Equivalently, we need to show  
\Be 
\label{max-ha}
\Big\| \sup_{t\in I_\circ} |U_t^{s_\ast}  f| \Big \|_{L^{q_\ast}} \le C\|f\|_{L^2},
\Ee
where 
\Be 
\label{def-u}
U_t^{s}  f:=  e^{it\sqrt{-\Delta}} (1-\Delta)^{-\frac s2} f=  (2\pi)^{-d} \int_{\mathbb R^n} e^{i(x\cdot \xi + t|\xi|)} \frac{\widehat f(\xi) d\xi}{ (1+|\xi|^2)^{s/2}} .
\Ee
 
 Let $P_j$ be the standard Littlewood-Paley projection operator defined by 
\[ \widehat{P_j f}(\xi)=\beta(|\xi|/2^j)  \widehat f(\xi), \quad j\ge 1,   \]  
where $\beta\in C_c^{\infty}(1/2, 2)$ such that $\sum_{j=-\infty}^\infty \beta(2^{-j} t)=1$ for $t>0$.  
We also set $P_0 =1- \sum_{j\ge 1} P_j f$.

\begin{lemma} 
\label{short} Let $s\in \mathbb R$, $q\ge 2$,  and $J$ be a unit interval. Then, 
\Be\label{low}
 \| \sup_{t\in J} |U_t^{s} P_0  f| \|_{L^{q}(\mathbb R^d)} \le C\|f\|_{L^2}. 
\Ee
\end{lemma}

\begin{proof} Recall \eqref{def-u}. By time translation and Plancherel's theorem, we may assume that $J=(0,1)$.  By the Cauchy--Schwarz inequality and 
Plancherel's theorem,  we have $\|U_t^{s} P_0  f\|_\infty \lesssim \|\widehat f\|_2=(2\pi)^{d/2}\|f\|_2$ for any $s\in \mathbb R$. Thus, it 
suffices to show \eqref{low} for $q=2$. Since $U_t^{s} P_0  f=\int_0^t \partial_\tau U_\tau^{s} P_0  f d\tau+ U_0^{s} P_0  f$, 
\[   \sup_{t\in (0,1)} |U_t^{s} P_0  f(x)| \le    \|  \partial_t U_t^{s} P_0  f(x)\|_{L_{t}^2(0,1)}+    |U_0^{s} P_0  f(x)|.\]
Note that  $|\partial_t U_t^{s} P_0  f(x)|=| U_t^{s} (-\Delta)^\frac12  P_0  f(x)|$. Thus, taking $L_x^2$ norm on both sides of the above inequality, we get \eqref{low} for $q=2$ since 
$\|\partial_t U_t^{s} P_0  f\|_{L_x^2}\lesssim \|f\|_2$ for any $t, s\in \mathbb R$ as can be easily seen by Plancherel's theorem. 
\end{proof}

Thanks to the estimate  \eqref{low}, it is sufficient to consider $\sum_{j\ge 1}U_t^{s_\ast} P_j f$. Now, from the estimate \eqref{max-frequency}  we have
\Be \label{basic}  \Big\|  \sup_{t\in I_\circ} |U_t^{s_\ast}  P_j f | \Big \|_{L^{q}( \mathbb R^d)}  \le C2^{(\frac {d+1}{2p}-\frac{d-1}{2q}-s_\ast)j} \|f\|_{L^p}\Ee
provided that $(1/p, 1/q)$ is contained in $\mathcal T_u\setminus \{ \mathfrak Q\} $ (see Figure \ref{fig2}).  Using these estimates, as before, we can show 
\Be 
\label{res-weak}
 \Big \|  \sum_{j\ge 1} \sup_{t\in I_\circ} |U_t^{s_\ast}  P_j f | \Big\|_{L^{q, \infty}( \mathbb R^d)}  \le C
\|f\|_{L^{p,1}}\Ee
whenever $(1/p, 1/q)\in  \operatorname{ int } \mathcal T_u$ and $ \frac {d+1}{2p}-\frac{d-1}{2q}=s_\ast.$  
Indeed,  we choose pairs $(p_\ell, q_\ell)$, $\ell=0,1$,  such that  $(1/p_\ell, 1/q_\ell)\in  \operatorname{ int } \mathcal T_u$ and 
\[ \frac {d+1}{2p_0}-\frac{d-1}{2q_0}-s_\ast<0< \frac {d+1}{2p_1}-\frac{d-1}{2q_1}-s_\ast.\]
Consequently, from \eqref{basic} we have two estimates for $(p_\ell, q_\ell)$, $\ell=0,1$. 
Applying  Lemma \ref{B-Sum-Tri} with $n=1$ to those two estimates gives \eqref{res-weak}.

Real interpolation among the consequent estimates \eqref{res-weak} upgrades those estimates  to the strong bounds
\[ \Big \|  \sum_{j\ge 1} \sup_{t\in I_\circ} |U_t^{s_\ast}  P_j f | \Big\|_{L^{q}( \mathbb R^d)}  \le C 
\|f\|_{L^p},\]
provided that $(1/p, 1/q)\in  \operatorname{ int } \mathcal T_u$ and $ \frac {d+1}{2p}-\frac{d-1}{2q}=s_\ast.$ 
Note that $(1/2, q_\ast)\in \operatorname{ int } \mathcal T_u$.  Therefore,  taking $p=2$ and $q=q_\ast$ 
gives the desired \eqref{max-ha}.

\subsection{The case for $q\in  (q_\circ(d), \infty)$} This case can be similarly handled as before. We let $q_\ast\in (q_\circ(d), \infty)$ and 
\[s_\ast= s_c(q_\ast, d)=d/2-d/q_\ast.\]   Then, the estimate \eqref{max-frequency} gives
\Be 
\label{basic2} \Big \|  \sup_{t\in I_\circ} |U_t^{s_\ast}  P_j f | \Big \|_{L^{q}( \mathbb R^d)}  \le C2^{(\frac{d-1}2+  \frac {1}{p}-\frac{d}{q}-s_\ast)j} \|f\|_{L^p}
\Ee
provided that $(1/p, 1/q)$ is contained in $\mathcal T_u\setminus \{ \mathfrak Q\} $ 
(see Figure \ref{fig2}
). 
Choosing suitable pairs $(p_\ell, q_\ell)$, $\ell=0,1$,  by \eqref{basic2} and Lemma \ref{B-Sum-Tri} with $n=1$,  
we have  
\Be
\label{sss} \Big \|  \sum_{j\ge 1} \sup_{t\in I_\circ} |U_t^{s_\ast}  P_j f | \Big\|_{L^{q, \infty}( \mathbb R^d)}  \le C
\|f\|_{L^{p,1}}\Ee
provided that $(1/p, 1/q)\in  \operatorname{ int } \mathcal T_l$ and $ \frac{d-1}2+  \frac {1}{p}-\frac{d}{q}=s_\ast.$ Real interpolation between those estimates  and   taking $p=2$ and $q=q_\ast$  
give the desired \eqref{max-ha}.    

\begin{rmk} 
\label{lolo} 
By real interpolation, we have a slightly stronger estimate
\Be 
\label{L2-h}
\Big\| \sup_{t \in (0,1)} \big|e^{it\sqrt{-\Delta}}f\big|\Big\|_{L^{q,2}(\mathbb{R}^d)} \le C_{d,q,s} \|f\|_{{H}^{s_c(q,d)}(\mathbb{R}^d)} 
\Ee
for $q\in(2, \infty)\setminus\{q_\circ(d)\} $.   Indeed, 
by the estimate in Lemma \ref{short} and real interpolation, we have 
$\| \sup_{t\in I_\circ} |U_t^{s} P_0  f| \|_{L^{q,2}(\mathbb R^d)} \le C\|f\|_{L^2}$ for $2\le q<\infty$ and $s\in \mathbb R$. 
Fixing $s_\ast=s_c(q,d)$ and using the estimates \eqref{res-weak} and \eqref{sss}, by  
interpolation between those estimates,   we also have 
\[\|  \sup_{t\in I_\circ} | \sum_{j\ge 1} U_t^{s_c(q,d)}  P_j f | \|_{L^{q,2}( \mathbb R^d)}  \le C
\|f\|_{L^{2}}\] for  $q\in(2, \infty)\setminus\{q_\circ(d)\} $.  Therefore, \eqref{L2-h} follows. 
\end{rmk}

\subsection{The case $q=q_\circ(d)$: Proof of \eqref{max2}}
\label{caseqc} Finally, we consider the case $q=q_\circ(d)$ and prove \eqref{max2}.  For the purpose, we make use of the estimate \eqref{max-frequency} 
with $(1/p,1/q)\in L_1$ for  $ \frac{2(d^2+2d-1)}{(d-1)(d+1)}< q\le \infty$ 
(see Figure \ref{fig2}).   In fact, we have  the estimate
\Be
\label{max-frequency2}
\| \sup_{t\in I_\circ}  |e^{it\sqrt{-\Delta}} f|\|_{L^{q}( \mathbb R^d)}  \le CN^{\frac{d+1}2- \frac{d^2+1}{(d-1)q}   } \|f\|_{L^p}
\Ee
for $p,q$ satisfying $(d-1)(1-\frac{1}{p}) =\frac{d+1}{q}$ and  $\frac{2(d^2+2d-1)}{(d-1)(d+1)}< q\le \infty$ whenever $\supp \widehat f \subset \mathbb A_N$.

Let $s_\ast=\tfrac d2-\tfrac d{q_\circ(d)}$.  Note that 
$ \tfrac{d+1}2- \tfrac{d^2+1}{(d-1)q}= s_\ast 
+ \frac{d^2+1}{d+1}\big(\tfrac1p-\tfrac12\big).$  Thus, by \eqref{max-frequency2}
 \[ \|  \sup_{t\in I_\circ} |U_t^{s_\ast}  P_j f | \Big \|_{L^{q}( \mathbb R^d)}  \le C 2^{j\frac{d^2+1}{d+1}(\tfrac1p-\tfrac12)} \|f\|_{L^p}\]
for $p,q$ satisfying $(d-1)(1-\frac{1}{p}) =\frac{d+1}{q}$ and  $\frac{2(d^2+2d-1)}{(d-1)(d+1)}< q\le \infty$. Therefore, as before, we may apply  Lemma \ref{B-Sum-Tri} with $n=1$ choosing pairs $(1/p_\ell, 1/q_\ell)\in L_1$  such that  $p_0<2<p_1$ and $ \frac{2(d^2+2d-1)}{(d-1)(d+1)}< q_0, q_1 \le \infty$. Consequently,  we obtain  
\[ \Big \|  \sum_{j\ge 1} \sup_{t\in I_\circ} |U_t^{s_\ast}  P_j f | \Big\|_{L^{q_\circ(d), \infty}( \mathbb R^d)}  \le C\|f\|_{L^{2,1}}.\]  
Combining this and \eqref{low} gives $ \| \sup_{t\in (1,2)} |U_t^{s_\ast}  f | \|_{L^{q_\circ(d), \infty}( \mathbb R^d)}  \le C
\|f\|_{L^{2,1}}$, which is equivalent to \eqref{max2}. 

Since $ \sup_{t\in (0,1)} |U_t^{s_\ast}  f |=\sup_{t\in (1,2)} |U_t^{s_\ast}  e^{-i\sqrt{-\Delta}} f|$,  we have
\[ \Big \| \sup_{t\in (0,1)} |U_t^{s_\ast}  f | \Big\|_{L^{q_\circ(d), \infty}( \mathbb R^d)}  \le C
\Big\|e^{-i\sqrt{-\Delta}} f\Big\|_{L^{2,1}}.\]
However, it seems unlikely that  $\|e^{-i\sqrt{-\Delta}} f\|_{L^{2,1}} $ can be bounded by $\| f\|_{L^{2,1}}$.

\subsection{Generalization to the operator $e^{it\phi(D)}$} 
\label{remark:b} 
In this subsection, we are concerned with the generalization $\bf{(G2)}$ in the introduction.
In the proof of Theorem \ref{thm:max},  the estimate \eqref{smoothing} is the key ingredient. 
To conclude that the same endpoint estimates  hold for $e^{it\phi(D)}$, it is sufficient to show 
 the estimate \eqref{smoothing} for $e^{it\phi(D)}$ replacing $e^{it\sqrt{-\Delta}}=e^{it|D|}$, since the rest of the argument in the above proof works without modification.

Let $\tilde {\mathbb A}_1=\{ \xi: 2^{-2}\le |\xi|\le 2^2\}$. For $N\ge 1$, we set   
 \[ \phi_N(\xi)= N^{-1}\phi(N \xi).\] 
We show that the same estimates remain valid for $e^{it\phi(D)}$
provided that 
\Be\label{g2g20}  |\partial^\alpha (\phi_N(\xi)-|\xi|)| \le CN^{-1}\Ee 
 for some constant $C>0$  if $\xi \in \tilde {\mathbb A}_1$ and $|\alpha|\le d+1$.   
 It is easy to check that   \eqref{g2g2}  implies  \eqref{g2g20}. 
 
 Let us set $\tilde\beta(p,q)=\beta(p,q)-\frac{d+1}q+\frac dp$. Let $U_t=e^{it\phi(D)}$ or $e^{it|D|}= e^{it\sqrt{-\Delta}}$.  Consider the estimate 
$
\| U_t f\|_{L^{q}( \mathbb R^d\times I_\circ)}  \le CN^{\beta(p,q)}  \|f\|_{L^p}
$
when $\supp \widehat f\subset \mathbb A_N$.   By scaling, it follows that  those  estimates  are respectively  equivalent to 
\Be
\label{smoothing3}
\| U_t f\|_{L^{q}( \mathbb R^d\times N I_\circ)}  \le CN^{\tilde\beta(p,q)} \|f\|_{L^p}
\Ee 
with   $U_t  =e^{it\phi_N(D)}$ and $U_t  =e^{it|D|}$  when $\supp \widehat f\subset \mathbb A_1$. For our purpose, we only need to show that the estimate \eqref{smoothing3} for $U_t=e^{it\phi_N(D)}$ 
follows from the corresponding estimate for $e^{it|D|}$ when $\supp \widehat f\subset \mathbb A_1$. The converse is also true, as becomes clear below. 

Let $\tilde\beta \in C_c^\infty(2^{-2}, 2^2)$ such that $\tilde\beta=1$ on $[2^{-1}, 2]$.  Let us set 
\[ a(\xi,t)=  (2\pi)^{-d}  e^{it( \phi_N(\xi)-|\xi|)}\tilde\beta(|\xi|).\]
Since $\supp \widehat f\subset \mathbb A_1$, we note that 
\[  e^{it\phi_N(D)} f(x)=   \int  e^{i(x\cdot\xi +t|\xi|)} a(\xi,t) \widehat f(\xi)\, d\xi.  \]
From \eqref{g2g20} it follows that $|\partial_\xi^\alpha a(\cdot,t)|\le C$ for $|\alpha|\le d+1$ and $|t|\le 2N$. Thus, expanding $a(\xi,t)$ in the Fourier series
on $[-2\pi, 2\pi]^d$ gives $a(\xi,t)=\sum_{\mathbf k\in \mathbb Z^d} c_{\mathbf k, t} e^{i\xi\cdot 2^{-1}\mathbf k} $ with 
$|c_{\mathbf k, t}|\le C (1+|\mathbf k|)^{-d-1}$ for  all $|t|\le 2N$. Consequently, it follows that 
\[ e^{it\phi_N(D)} f(x)=\sum_{\mathbf k\in \mathbb Z^d}   c_{\mathbf k, t} e^{it|D|} f(x+2^{-1} \mathbf k)\]
for $|t|\le 2N$. Therefore, the estimate \eqref{smoothing3} for $U_t=e^{it\phi_N(D)}$ 
follows from that for $e^{it|D|}$.  The converse can also be seen in the same manner. 

\section{Proof of Propositions \ref{prop:xlocal to global}, \ref{prop:local to global}, and \ref{prop:local-global to global}} 

In this section we prove the three propositions \ref{prop:xlocal to global}, \ref{prop:local to global}, and \ref{prop:local-global to global}, which concern the equivalence  between the maximal estimates  \eqref{max}, \eqref{local-local}, \eqref{global-local}, and \eqref{global-global0}.   

We begin by noting  that 
\Be 
\label{uts}
U_t^{s}  f(x)= \int_{\mathbb R^d} K(x-y,t) f(y) dy,
\Ee 
where 
\begin{equation}
\label{KK}
K(x,t) = (2\pi)^{-d} \int_{\mathbb R^d} e^{i(x\cdot \xi + t|\xi|)} \frac{d\xi}{ (1+|\xi|^2)^{s/2}}.
\end{equation}

Recall $\beta$ from Section 3 and let $\beta_0=1-\sum_{j=1}^{\infty}\beta(2^{-j}\cdot)$. We set
\begin{align}
\label{kj}
K_j(x,t)&= (2\pi)^{-d}  \int e^{i(x\cdot \xi + t|\xi|)} \frac{ \beta(2^{-j}|\xi|) d\xi}{ (1+|\xi|^2)^{s/2}}, \quad j\ge 1, 
\\
\label{k0}
K_0(x,t)&=(2\pi)^{-d}  \int  e^{i(x\cdot \xi + t|\xi|)}   \frac{\beta_0(|\xi|) d\xi}{ (1+|\xi|^2)^{s/2}}.
\end{align}
Thus, we  have
\begin{equation*}\label{ksum}
K = \sum_{j=0}^{\infty}K_j. 
\end{equation*}

Henceforth, for simplicity, we set 
\[ I=(0,1), \quad \B=B^d(0,1).\]

\subsection{Proof of Proposition \ref{prop:xlocal to global}}
Since the implication from \eqref{max} to \eqref{local-local} is trivial, we only need to show that  \eqref{local-local} implies \eqref{max}. 
Recalling \eqref{def-u}, we note that \eqref{max} is equivalent to
\[\big\|  \sup_{t\in I} \big|U_t^s f\big|\big\|_{L^{q}(\mathbb R^d)} \le C\|f\|_{2}.\]

Let us set 
\[\textstyle \tilde U_t^s f= \sum_{j\ge 1} U_t^s P_j f.\]
Thanks to Lemma \ref{short}, it is sufficient to show that  
\begin{equation*}\label{tilde-global} \big\|  \sup_{t\in I} \big| \tilde U_t^s f\big|\big\|_{L^{q}(\mathbb R^d)} \le C\|f\|_{2}
\end{equation*}
while assuming that \eqref{local-local} holds, that is to say, 
\Be \label{tilde-local} \big\|  \sup_{t\in I} \big| \tilde U_t^s f\big|\big\|_{L^{q}(\mathbb B)} \le C\|f\|_{2}.\Ee

Let $\mathcal Q=\{Q\}$ be a collection of almost disjoint unit cubes that covers $\mathbb R^d$.  Let $\bar Q$ denote the cube of side length $5$ that has the same center as $Q$ so that $\dist(Q, \bar Q^c)\ge 2$. 
We have
\[   \big\|  \sup_{t\in I} \big| \tilde U_t^s f\big|\big\|_{L^{q}(\mathbb R^d)}  = \Big( \sum_{Q\in \mathcal Q} \big\|  \sup_{t\in I}  \big| \tilde U_t^s f \big|\big\|_{L^{q}(Q)}^q\Big)^{1/q}
\le  2(\mathrm I + \mathrm I\! \mathrm I),\]
where 
\[  \mathrm I^q= \sum_{Q\in \mathcal Q} \big\|  \sup_{t\in I}  \big| \tilde U_t^s f_{\bar Q} \big|\big\|_{L^{q}(Q)}^q,   \quad \mathrm I\! \mathrm I^q = \sum_{Q\in \mathcal Q} \big\|  \sup_{t\in I}  \big| \tilde U_t^s f_{\bar Q^c} \big|\big\|_{L^{q}(Q)}^q. \] 
We will show that 
\Be\label{conclusion}     \mathrm I, \,  \mathrm I\! \mathrm I   \lesssim   \|f\|_2.   \Ee 

By translation,  from \eqref{tilde-local} it follows that  $\big\|  \sup_{t\in I}  \big| \tilde U_t^s f_{\bar Q} \big|\big\|_{L^{q}(Q)}\lesssim  \|f_{\bar Q}\|_2$ for all $Q\in \mathcal Q$. Thus, we have
\[ \mathrm I \lesssim \Big( \sum_{Q\in \mathcal Q}  \|f_{\bar Q}\|_2^q\Big)^{1/q}\le \Big \| \Big( \sum_{Q\in \mathcal Q} |f_{\bar Q}|^q\Big)^{1/q} \Big \|_2 \lesssim \|f\|_2. \]

To handle $\mathrm I\! \mathrm I$, we use an estimate for the kernel of $\tilde U_t^s$. 
Let us set  $\tilde K(x,t) =  \sum_{j=1}^{\infty}K_j(x,t)$. 
 From \eqref{uts} and \eqref{kj} it is clear that 
\Be
 \label{tilde-u}  \tilde U_t^s  f(x)=   \int_{\mathbb R^d} \tilde K(x-y,t) f(y) dy.
\Ee
We claim that 
\Be 
\label{tilde-est}  |\tilde K(x, t)|\le \mathcal E_N(x):= C(1+ |x|)^{-N}
\Ee 
for any $N$ if $t\in I$ and $|x|\ge 2$.  Indeed, by scaling we have 
\[ K_j(x,t)= (2\pi)^{-d}  2^{dj} \int e^{i2^j(x\cdot \xi + t|\xi|)} \frac{ \beta(|\xi|) d\xi}{ (1+2^{2j}|\xi|^2)^{s}}.\]
Since $t\in I$ and $|x|\ge 2$, $|\nabla_{\xi} (x\cdot \xi + t|\xi|)|\ge |x|/2$. Thus, routine integration by parts gives
$|K_j(x,t)|\lesssim 2^{(d-2s)j} (2^j|x|)^{-N}$ for any $N$ if $t\in I$ and $|x|\ge 2$. Therefore, summation along $j\ge 1$ gives the estimate \eqref{tilde-est}. 

Since $\dist(Q, \bar Q^c)\ge 2$, using \eqref{tilde-u} and \eqref{tilde-est},  we have 
\[   \sup_{t\in I}  |\tilde U_t^s f_{\bar Q^c} (x)|\lesssim  \mathcal E_N\ast |f|(x)   \]
for $x\in Q$.  Therefore, 
\[  \mathrm I\! \mathrm I\lesssim   \|\mathcal E_N\ast |f|\|_q.\]
By choosing $N\ge d+1$, we have $\mathcal E_N\in L^1\cap L^\infty$. Thus,  $\|\mathcal E_N\ast |f|\|_q\lesssim \|f\|_p$ for any $p\le q$. In particular, we have 
$\mathrm I\! \mathrm I\lesssim \|f\|_2$. Therefore, we get \eqref{conclusion}. This completes the proof.

\begin{rmk}
\label{rmk:g3g3}   In order to extend Proposition \ref{prop:xlocal to global} to the  general operator $e^{it\phi(D)}$, we consider the  kernel 
of the operator $ (1+|D|^2)^{-s/2}e^{it\phi(D)}$, which is given by  
\begin{equation*}
\label{k-phi0} K^\phi(x,t) = (2\pi)^{-d} \int_{\mathbb R^d} e^{i(x\cdot \xi + t\phi(\xi))} \frac{d\xi}{ (1+|\xi|^2)^{s/2}}.
\end{equation*}
From the perspective of the  proof in the above, it is enough to show that  
\Be
\label{k-phi} |K^\phi (x, t)|\lesssim (1+ |x|)^{-d-1}
\Ee for $|x|\ge C$ with 
a sufficiently large 
$C>0$.  By  Littlewood-Paley decomposition and scaling  it follows that $K^\phi=\sum_{j=0}^\infty K^\phi_j$, where
\begin{equation*}
\label{phi-kj} K^\phi_j(x,t)= (2\pi)^{-d}  2^{dj} \int e^{i2^j(x\cdot \xi + t\phi_{2^j}(\xi))} \frac{ \beta(|\xi|) d\xi}{ (1+2^{2j}|\xi|^2)^{s}}.
\end{equation*}

Now, we note  from \eqref{g3} that there are constants $C, N_0>0$ such that 
\begin{equation*}\label{phi-up}
 |\partial^\alpha \phi_N(\xi)| \le C
 \end{equation*}
for  $\xi\in \tilde{\mathbb A}_1$  if $N\ge N_0$ and $|\alpha|\le d+1$.  
Thus, we have  $|\nabla_{\xi} (x\cdot \xi + t\phi_{2^j}(\xi))|\ge |x|/2$ for  $\xi \in \tilde{\mathbb A}_1$ if $|x|\ge \tilde C$ for a large enough $\tilde C>0$. Routine integration by parts gives
$|K_j^\phi(x,t)|\lesssim 2^{(d-2s)j} (2^j|x|)^{-d-1}$  if $t\in I$,  $|x|\ge \tilde C$ and $2^j\ge 2N_0$. The same bounds trivially hold for $0\le j< \log  2N_0$ since the 
kernels are Schwartz functions. Therefore, 
summation over $j$ gives the desired estimate \eqref{k-phi}.

\end{rmk}

\subsection{Proof of Proposition \ref{prop:local to global}}
We only need to show the implication from \eqref{local-local} to \eqref{global-local} since the converse is trivial.

 For our purpose,  recalling \eqref{def-u},   
it is sufficient  to show 
\begin{equation}\label{1}
\big\| \sup_{t \in \mathbb R} \big|U_t^s f\big|\big\|_{L^{q}(\B)} \le C \|f\|_{L^2},
\end{equation}
which is equivalent to  \eqref{global-local}, 
while assuming that 
\begin{equation}\label{2}
\big\| \sup_{t \in I} \big|U_t^s f\big|\big\|_{L^{q}(\B)} \le C \|f\|_{L^2}.
\end{equation}
The last estimate \eqref{2} is clearly equivalent to \eqref{local-local}.  We begin by observing that  there is a constant $C$, independent of $J$, such that 
\begin{equation}
\label{3}
	\big\| \sup_{t \in J} \big|U_t^s f\big|\big\|_{L^{q}(\B)} \le C \|f\|_{L^2}
\end{equation}
holds for any  unit interval $J$ with a constant $C$.  This is clear from \eqref{2} since   $e^{it\sqrt{-\Delta}}$ is a unitary operator.  

We now consider the dual forms of the estimates \eqref{1} and \eqref{3}. 
Let  us set $T=U_t^s$ and let $T^*$ denote the adjoint operator of $T$. It is easy to see that
\[
T^*F(x) = \int (1-\Delta)^{-s/2}e^{-it'\sqrt{-\Delta}}  F(\cdot, t') \, dt'.
\]

For simpler notation, we set 
\begin{align*}
g_{\B}(x,t) &= \chi_\B(x)g(x,t),
\\
g_{\B,J}(x,t) &= g_{\B}(x,t)\chi_{J}(t).
\end{align*}
 By duality, it suffices to show that 
 \begin{equation}\label{global maximal}
\big\|\,T^* g_{\B} \big\|_{L^2} \le C \| g\|_{L_x^{q'}L_t^1}
\end{equation}
holds provided that 
 \begin{equation}\label{local maximal}
\big\|\, T^* g_{\B,J} \big\|_{L^2} \le C \|g\|_{L_x^{q'}L_t^1}
\end{equation}
holds  for arbitrary unit interval $J$. It is clear  that  \eqref{global maximal} and \eqref{local maximal} are equivalent 
to \eqref{1} and \eqref{3}, respectively.

We proceed to show \eqref{global maximal}. Let $\mathfrak J=\{J\}$ be a collection of   almost disjoint unit intervals $J$  such that 
\[\mathbb R = \bigcup_{J\in \mathfrak J} J.\]  
Consequently, $g_{\B}=\sum_{J} g_{\B,J}$. Thus, \eqref{global maximal} follows if we show 
\begin{equation*}
\Big|\sum_{J, J'} \big\langle T^*g_{\B,J}, T^*g_{\B,J'}\big\rangle \Big| \le C \|g_{\B} \|_{L_x^{q'}L_t^1}^2.
\end{equation*}

To this end, we divide 
\begin{equation*}
\sum_{J, J'} \big\langle T^*g_{\B,J}, T^*g_{\B,J'}\big\rangle=\mathcal{I}_1 + \mathcal{I}_2,
\end{equation*}
where
\begin{align*}
\mathcal{I}_1&=\sum_{J, J':\dist(J, J')< 4}\big\langle T^*g_{\B,J}, T^*g_{\B,J'}\big\rangle,\\
\mathcal{I}_2&=\sum_{J,J':\dist(J, J')\ge 4} \big\langle T^*g_{\B,J}, T^*g_{\B,J'}\big\rangle. 
\end{align*}

We first consider $\mathcal{I}_1$. By the Cauchy-Schwarz inequality and the estimate \eqref{local maximal}
it follows that $|\big\langle T^*g_{\B,J}, T^*g_{\B,J'}\big\rangle| \le  \|  T^*g_{\B,J}\|_2\| T^*g_{\B,J'}\|_2\lesssim \| g_{\B,J}\|_{L_x^{q'}L_t^1} \| g_{\B,J'}\|_{L_x^{q'}L_t^1}$. Hence, 
\[ |\mathcal{I}_1|\lesssim \sum_{J,J': \dist(J, J')< 4}  \| g_{\B,J}\|_{L_x^{q'}L_t^1} \| g_{\B,J'}\|_{L_x^{q'}L_t^1}.\]
The Cauchy--Schwarz inequality gives
\begin{align*}
|\mathcal{I}_1| 
& \lesssim \Big(\sum_{J} \|g_{\B,J}\|_{L_x^{q'}L_t^1}^2  \Big)^{1/2} 
\Big(\sum_{J'} 
\|g_{\B,J'}\|^2_{L_x^{q'}L_t^1}\Big)^{1/2}.
\end{align*}
Since $1 \le q'\le 2$ and the intervals $J$ are almost disjoint, Minkowski's inequality gives
\Be
\label{easy} 
\Big(\sum_J \|g_{\B,J}\|_{L_x^{q'}L_t^1}^2\Big)^\frac12 \le \big\|\big(\sum_J |g_{\B,J}|^2\big)^{1/2}\big\|_{L_x^{q'}L_t^1} \lesssim \big\|\sum_J g_{\B,J}\big\|_{L_x^{q'}L_t^1}.   \Ee 
Therefore, we obtain 
\begin{equation*}
     |\mathcal{I}_1|  \lesssim \|g_{\B}\|_{L_x^{q'}L_t^1}^2.
\end{equation*}

We now show that $\mathcal{I}_2$ also has the same upper bound. To this end, we write 
\[ \big\langle T^*g_{\B,J}, T^*g_{\B,J'}\big\rangle = \big\langle g_{\B,J}, TT^* g_{\B, J'}\big\rangle\]
and make use of an estimate for the kernel estimates of $TT^*$. 
We note that 
\begin{equation}
\label{ker-}
TT^*h(x,t)=\iint  K(x-y,t-t') h (y,t') dy dt', 
\end{equation}
where $K$ is given by \eqref{KK}. 
To obtain the desired estimate for $\mathcal{I}_2$,  the following estimate for $K$  is crucial. 
\begin{lemma}\label{kernel}
If $|t|\ge 4$ and $|x|\le 2$, then we have the estimate
 \[ |K(x,t)|\lesssim |t|^{-d}.\]
\end{lemma}

\begin{proof}[Proof of Lemma \ref{kernel}]
Recalling \eqref{kj} and \eqref{k0}, we write the integral in the spherical coordinates and  make a change of variables to get 
\[  K_j(x,t)= (2\pi)^{-d}  \int_{\mathbb S^{d-1}} K^\theta_j(x,t) d\theta, \]  
where 
\begin{equation*}
K_j^\theta(x,t)=2^{dj} \int e^{i2^jr(x\cdot \theta + t)}  \frac{\beta(r)  r^{d-1} dr}{(1+2^{2j}r^2)^{s}}. 
\end{equation*} 

Let $a_j(r)={\beta(r)  r^{d-1}}{(1+2^{2j}r^2)^{-s}}$. We now note that  
\begin{equation*}
K_j^\theta(x,t)= 2^{dj} \int \Big[\Big( \frac{1}{2^j(x\cdot\theta+t)} \frac{d}{dr}  \Big)^k e^{i2^jr(x\cdot \theta + t)}\Big] a_j(r) dr .
\end{equation*}
Also, note that  $|x\cdot\theta+t|\ge |t|/2$ since $|t|\ge 4$ and $|x|\le 2$. Using the fact that   $a_j^{(k)}=O(2^{-2sj})$ for any $k\ge 0$, by routine integration by parts we obtain  $|K_j^\theta(x,t)| \le C2^{(d-2s)j}(2^{j}|t|)^{-k}$ for any $k \ge 0$. 
 Taking arbitrarily large $k$, we obtain 
$
\sum_{j>0}|K_j^\theta(x,t)| \le C|t|^{-N}
$
for any $N$. Consequently, integration over the sphere $\mathbb S^{d-1}$ gives
\begin{equation}
\label{sum}
\sum_{j>0}|K_j(x,t)| \le C|t|^{-N}
\end{equation}
for any $N$.

We now consider the case $j=0$, which requires additional care since $\xi\mapsto |\xi|$ is not smooth at the origin.  
Note that $\beta_0(r)=0$ if $r\ge 2$. 
As before, we have $K_0(x,t)= (2\pi)^{-d}  \int_{\mathbb S^{d-1}} K^\theta_0(x,t) d\theta,$
where 
\begin{equation*}
K_0^\theta(x,t)= \int_0^2 e^{ir(x\cdot \theta + t)} a_0(r) dr
\end{equation*} 
and $a_0(r)=\beta_0(r)  r^{d-1}{(1+r^2)^{-s}}$.  Note that  $a_0^{(k)}(0)=0$ for $k=0, \dots, d-2$ and  $a_0^{(k)}(2)=0$ for all $k$.   
By integration by parts $d$ times as above (it should be noted that $a_0^{(d-1)}(0)\neq 0$),  we have
\[|K_0^\theta| \lesssim |t|^{-d}.\]
Integrating the estimates over the unit sphere yields  the bound $|K_0(x,t)|\lesssim |t|^{-d}.$ 

Therefore, combining this and the estimate \eqref{sum} gives the desired estimate. 
\end{proof}

By \eqref{ker-} and  Lemma  \ref{kernel},  
we have
\[ |\chi_{\B\times J}TT^* (\chi_{\B\times J'} h)(x,t)|\lesssim  \chi_\B(x)  \iint \chi_{J}(t) |t-t'|^{-d}  \chi_{J'}(t') \chi_\B(y) |h (y,t')| dy dt'\]
   provided that $\dist(J, J')\ge 4$. Thus, when  $\dist(J, J')\ge 4$, it follows that
   \[\|\chi_{\B\times J}TT^* (\chi_{\B\times J'} h)(x,\cdot)\|_{L_t^\infty} \lesssim \dist(J, J')^{-d}   \iint  \chi_{\B\times \B}(x,y) \|h (y,\cdot)\|_{L_t^1} dy. \]   Therefore,  we have
\begin{equation}\label{$TT^*$}
\|\chi_{\B\times J}TT^* \chi_{\B\times J'} h \|_{L_x^r L_t^\infty} \le C(1+\dist(J, J'))^{-d}\| h\|_{L_x^{p}L_t^{1}}
\end{equation}
 for any $1\le p\le r\le \infty$ if  $\dist(J, J')\ge 4$.  
Since $\big\langle T^*g_{\B,J}, T^*g_{\B,J'}\big\rangle = \big\langle g_{\B,J}, TT^* g_{\B,J'}\big\rangle$, 
\begin{align*}
 |\mathcal{I}_2|&\le  \sum_{J,J': \dist(J, J')\ge 4} \big|\big\langle g_{\B,J}, TT^* g_{\B,J'}\big\rangle \big|
 \\
&\le  \sum_{J,J': \dist(J, J')\ge 4} \| g_{\B,J}\|_{L_x^{q'}L_t^1} \|TT^* g_{\B,J'}\|_{L_x^{q}L_t^\infty}.
\end{align*}
Using the estimate \eqref{$TT^*$} for $p=q'$ and $r=q$ gives
\begin{align*}
 |\mathcal{I}_2|\le C \sum_{J,J'} (1+\dist(J, J'))^{-d} \|g_{\B,J}\|_{L_x^{q'}L_t^1} \|g_{\B,J'}\|_{L_x^{q'}L_t^1}.
\end{align*}
Note   $
 \sum_{J} (1+\dist(J, J'))^{-d} \le C
$ for any $J'$ and some constant $C>0$. The same also holds by replacing the roles of $J$ and $J'$. Therefore,  Schur's test gives
\begin{equation*}
 |\mathcal{I}_2|\le C \sum_{J}\|g_{\B,J}\|_{L_x^{q'}L_t^1}^2.
\end{equation*}
Consequently, using \eqref{easy},   we conclude that  $|\mathcal{I}_2|\le C\|g_B\|_{L_x^{q'}L_t^1}^2.$
This completes the proof.

\subsection{Proof of Proposition \ref{prop:local-global to global}}
\label{proof-scaling}
 Since $s_c(q, d)\ge 0$, the implication from  \eqref{global-global0} to \eqref{local-local} is trivial. Thus, we only need to prove that  \eqref{local-local} implies \eqref{global-global0}. 
From the inequality \eqref{local-local} and time reversal symmetry,  it follows that
\begin{equation}\label{Sobolev maximal estimate}
\Big\| \sup_{t \in (-1,1)} \Big|  \int e^{i(x \cdot \xi \, + t|\xi|)} \chi_{\{|\xi|\ge 1\}}\widehat{f}(\xi\,)|\xi|^{-s_c(q,d)} d\xi \Big|\Big\|_{L^{q}(\B)} \le C \|f\|_{L^2}
\end{equation}
for all $f\in L^2$. 
Let us set 
\[g(x) = R^{d/2} f(R x).\] 
Since $\|g\|_2=\|f\|_2$,  applying the inequality \eqref{Sobolev maximal estimate} to $g$ gives 
\Be 
\label{eee}
 \Big\|  R^{-\frac d2} \sup_{t \in (-1,1)} \Big|  \int e^{i(x \cdot \xi \, + t|\xi|)} \chi_{\{|\xi|\ge 1\}}\widehat{f}( \xi/R\,)|\xi|^{-s_c(q,d)} d\xi \Big|\Big\|_{L^{q}(\B)} \le C \|f\|_{L^2}.
 \Ee
 Since $s_c(q,d)=d/2-d/q$, changing  variables $\xi \rightarrow R\xi$,  we have
 \begin{align*}
 R^{-\frac d2} \sup_{t \in (-1,1)}   \Big|  \int &e^{i(x \cdot \xi \, + t|\xi|)} \chi_{\{|\xi|\ge 1\}}\widehat{f}(\xi/R\,)|\xi|^{-s_c(q, d)} d\xi \Big| 
\\
&= R^{\frac dq} \sup_{t \in (-R,R)} \Big|   \int e^{i(Rx \cdot \xi \, + t|\xi|)} \chi_{\{|\xi|\ge R^{-1}\}}\widehat{f}(\xi\,)|\xi|^{-s_c(q,d)} d\xi \Big|,
\end{align*}
Therefore,  combining this and \eqref{eee}, by the change of variables $x \rightarrow x/R$, we obtain 
\begin{align*}
\Big\| \sup_{t \in (-R,R)} \Big|  \int &e^{i(x \cdot \xi \, + t|\xi|)} \chi_{\{|\xi|\ge R^{-1}\}}\widehat{f}(\xi\,)|\xi|^{-s_c(q, d)} d\xi \Big|  \Big\|_{L^{q}(R\B)}\le C\|f\|_2.
\end{align*}
 Letting $R \rightarrow \infty$ yields 
\[ \| \sup_{t \in \mathbb R} \Big|  \int e^{i(x \cdot \xi \, + t|\xi|)} \widehat{f}(\xi\,)|\xi|^{-s_c(q, d)} d\xi \Big|  \Big\|_{L^{q}(\mathbb R^d)}\le C\|f\|_2,\]
which is clearly equivalent to   the global maximal estimate \eqref{global-global0} with $ s= s_c(q,d)$.

\section*{Acknowledgements}

This work was supported by the National Research Foundation of Korea (NRF) grants  no. RS-2023-00239774 (Cho) and no. RS-2024-00342160  (Lee), as well as the National Key R\&D Program of China grant no. 2023YFA1010800 and the Natural Science Foundation of China grant no. 12271435 (Li). 

\bibliographystyle{plain}

\end{document}